\documentclass[11pt]{article}
\usepackage{hyperref}
\usepackage{amsmath, amssymb, amsthm, mathtools, tikz, pgfplots, etoolbox}
\pgfplotsset{compat=1.14}
\usepackage[capitalize, noabbrev]{cleveref}

\usepackage{microtype, fullpage}
\author{
  Zilin Jiang\thanks{School of Mathematical and Statistical Sciences, and School of Computing and Augmented Intelligence, Arizona State University, Tempe, AZ 85281, USA. Email: {\tt zilinj@asu.edu}. Supported in part by an AMS Simons Travel Grant and by U.S. taxpayers through NSF Award DMS-1953946. Part of the work was done when Z.~Jiang was a postdoctoral fellow at Technion -- Israel Institute of Technology, and was supported in part by ISF grant nos 409/16, 936/16.}
  \and
  Amir Yehudayoff\thanks{Department of Mathematics, Technion -- Israel Institute of Technology, Technion City, Haifa 3200003, Israel. Email: {\tt amir.yehudayoff@gmail.com}. Supported in part by ISF grant no 1162/15.}
}

\title{An isoperimetric inequality for Hamming balls \\ and local expansion in hypercubes}
\date{}

\newtheorem{theorem}{Theorem}
\newtheorem{lemma}[theorem]{Lemma}
\newtheorem{proposition}[theorem]{Proposition}

\theoremstyle{remark}
\newtheorem*{remark}{Remark}

\newtheorem{myclaim}{Claim}
\Crefname{myclaim}{Claim}{Claims}

\newenvironment{claimproof}[1][Proof of Claim]{\begin{proof}[#1]}{\end{proof}}

\newcommand{\A}{\mathcal{A}}
\newcommand{\B}{\mathcal{B}}
\newcommand{\C}{\mathcal{C}}

\newcommand{\M}{\mathcal{M}}
\newcommand{\cS}{\mathcal{S}}

\newcommand{\F}{\mathbb{F}}
\newcommand{\N}{\mathbb{N}}
\newcommand{\R}{\mathbb{R}}

\newcommand{\hg}{\mathfrak{H}}

\newcommand{\abs}[1]{\left\lvert#1\right\rvert}
\newcommand{\dset}[2]{\left\{#1 \colon #2\right\}}
\newcommand{\sset}[1]{\left\{#1\right\}}

\newcommand{\bd}{\partial}
\newcommand{\bp}{\partial^+}
\newcommand{\bm}{\partial^-}
\newcommand{\ncr}{\binom{n}{r}}

\newcommand{\eps}{\varepsilon}

\DeclarePairedDelimiter{\floor}{\lfloor}{\rfloor}

\begin{document}

\maketitle

\begin{abstract}
We prove a vertex isoperimetric inequality for the $n$-dimensional Hamming ball $\mathcal{B}_n(R)$ of radius $R$. The isoperimetric inequality is sharp up to a constant factor for sets that are comparable to $\mathcal{B}_n(R)$ in size. A key step in the proof is a local expansion phenomenon in hypercubes.
\end{abstract}

\section{Introduction}

Isoperimetric inequalities allow to control the boundary size or surface area of bodies in terms of their volume. The classical isoperimetric inequality states that in Euclidean spaces, balls have the smallest surface area per given volume. Such inequalities are fundamental in geometry, and are deeply related to many areas of mathematics and physics. 

In this paper, we consider discrete spaces. For a graph $G = (V, E)$ and a subset $X$ of vertices, the \emph{vertex boundary}\footnote{Another interpretation of the term ``boundary'' for graphs is the edge boundary: the set of edges exiting $X$.} $\bd_G X$ of $X$ is the set of vertices in $V\setminus X$ which have a neighbor in $X$. The vertex isoperimetric problems for graphs concern the minimum possible vertex boundary size of $X$ given its size.

We focus on the vertex isoperimetric problem for Hamming balls. The $n$-dimensional \emph{Hamming ball} $B_n(r)$ of radius $r$ is the graph with vertex set $\B_n(r)$ consists of all subsets of $[n]$ of size at most $r$, and two subsets are adjacent if they differ by exactly one element.

We establish the following approximate isoperimetric inequality for Hamming balls.

\begin{theorem}[Isoperimetric inequality for Hamming balls] \label{thm:iso-ball}
  For every $\rho \in (0, 1/2)$, there is a positive integer $n_0$ so that the following holds. For every $n \geq n_0$, $R\leq n/2$, and $\A \subseteq \B_n(R)$, if
  \[
    \abs{\B_n(\floor{\rho n})} \le \abs{\A} \le \abs{\B_n(R)} - \abs{\B_n(\floor{\rho n})},
  \]
  then the vertex boundary of $\A$ in the Hamming ball $B_n(R)$ satisfies
  \[
    \abs{\bd_{B_n(R)}\A} \ge \frac{\rho^{3/2}}{18\sqrt{n}}\min\left(\abs{\A}, \abs{\B_n(R)\setminus\A}\right).
  \]
\end{theorem}

\cref{thm:iso-ball} is sharp up to a constant factor depending only on $\rho$ for $\A$ that are comparable to $\B_n(R)$ in size.

\begin{proposition} \label{prop:iso-ball-sharp}
  For every $\eps \in (0,1/2)$ and $n, R \in \N$ such that $\eps n \le R \le n / 2$, and for every $\alpha \in (\eps, 1-\eps)$, there exists $\M \subseteq \B_n(R)$ of size $\floor{\alpha \abs{\B_n(R)}}$ such that 
  \[
    \abs{\bd_{B_n(R)}\M} \le O_\eps\left(1/\sqrt n\right)\min(\abs{\M}, \abs{\B_n(R)\setminus \M}).
  \]
\end{proposition}

Our results are discrete analogs of an isoperimetric inequality in Gaussian space. To illustrate the analogy, we recall the following classical isoperimetric inequalities. The $n$-dimensional \emph{hypercube} $Q_n$ is the $n$-dimensional Hamming ball $B_n(n)$ of radius $n$.
\begin{enumerate}
  \item (The Gaussian isoperimetric inequality \cite{MR0365680,MR399402}) Among all sets of a given standard Gaussian measure in $\R^n$, half-spaces minimize the Gaussian boundary measure.
  \item (Harper's theorem \cite{harper1966optimal}) Among all vertex subsets of $Q_n$ of the size $\abs{\B_n(R)}$, the Hamming ball $\B_n(R)$ has the smallest vertex boundary in $Q_n$.
\end{enumerate}
Harper's theorem can be seen as a discrete analog of the Gaussian isoperimetric inequality. Indeed, by viewing a subset of $[n]$ as its indicating vector, the Hamming ball $\B_n(R)$ can be thought of as a half-space whose bounding hyperplane has a normal vector $v_1 := (1,\dots,1)$.

\begin{table}
  \centering
  \begin{tabular}{c|cc|cc}
    & \multicolumn{2}{c|}{Gaussian} & \multicolumn{2}{c}{Discrete} \\ \hline
    Space & $\R^n$ & $H$ & $Q_n$ & $B_n(R)$ \\
    Minimizer & $H$ & $H \cap M$ & $\B_n(R)$ & ?
  \end{tabular}
  \caption{Gaussian isoperimetric problems and their discrete analogs.}
  \label{table:analog}
\end{table}

We are concerned with the discrete space $B_n(R)$. Its Gaussian analog should concern a half-space $H \subseteq \R^n$ endowed with the conditional Gaussian measure. It is known that the minimizers of the boundary measure are sets of the form $H \cap M$ where $M$ is another half-space whose bounding hyperplane is perpendicular to that of $H$ (see \cite[Proposition~5.1]{lee2006isoperimetric}).

Which vertex subsets have the smallest vertex boundary in $B_n(R)$? From \cref{table:analog} the answer should be a discrete analog of $H \cap M$. \cref{thm:iso-ball,prop:iso-ball-sharp} answer this question approximately. Indeed one of the examples in \cref{prop:iso-ball-sharp} is defined by
\[
  \M := \dset{X \in \B_n(R)}{\abs{X \cap \sset{1, \dots, n/2}} \le \abs{X} / 2},
\]
which can be seen as the intersection of the half-space $\B_n(R)$ with another half-space whose bounding hyperplane has a normal vector $v_2 := (1,1,\ldots,1,-1,-1,\ldots,-1)$, where $v_2$ has equal number of $1$'s and $-1$'s. As in the Gaussian analog, the two normal vectors $v_1$ and $v_2$ are orthogonal.

The key ingredient in the proof of \cref{thm:iso-ball} is a local expansion statement for hypercubes concerning the \emph{lower shadow} $\bm_n \A$ and the \emph{upper shadow} $\bp_n \A$ of $\A \subseteq \cS_n(r)$ in $Q_n$ defined by
\[
  \bm_n \A := (\bd_n\A) \cap \cS_n(r-1)
  \quad \text{and} \quad
  \bp_n \A := (\bd_n\A) \cap \cS_n(r+1).
\]
To put it in context, we recall the \emph{normalized matching property} of hypercubes, which can be proved by a simple double counting argument.

\begin{proposition}[Normalized matching property] \label{lem:LYM}
  Suppose $r$ and $s$ are two positive integers and $n = r + s$. For every $\A\subseteq \cS_n(r)$, its lower and upper shadows satisfy
  \[
    \pushQED{\qed}
    \abs{\bm_n\A} \ge \frac{r}{s+1}\abs{\A}
    \quad\text{and}\quad
    \abs{\bp_n\A} \ge \frac{s}{r+1}\abs{\A}.
    \qedhere\popQED
  \]
\end{proposition}

Although \cref{lem:LYM} is much weaker than the Kruskal--Katona theorem~\cite{MR0154827,MR0290982} or a weak form due to Lov\'asz~\cite[Ex.~13.31(b)]{MR1265492}, the normalized matching property is essentially sharp. For example, the lower shadow of $\A_0 := \dset{X\in\cS_n(r)}{1\not\in X}$ has size $\frac{r}{s}\abs{\A_0}$, and the upper shadow of $\A_1 := \dset{X\in\cS_n(r)}{1\in X}$ has size $\frac{s}{r}\abs{\A_1}$.

The two sets $\A_0$ and $\A_1$ are very different. It is natural to ask if the two inequalities in \cref{lem:LYM} can be essentially sharp for the same $\A$. Certainly, when $\A = \varnothing$ or $\A = \cS_n(r)$, equalities hold for both inequalities. However, we dash the hopes of a non-trivial set that behaves like both $\A_0$ and $\A_1$. We abbreviate $\bd_{Q_n}$ by $\bd_n$ throughout the article.

\begin{theorem}[Local expansion] \label{thm:expLYM}
  Suppose $r,s$ are two positive integers and $n = r + s$. For every $\A \subseteq \cS_n(r)$ of size $\alpha\ncr$, the vertex boundary of $\A$ in $Q_n$ satisfies
  \begin{equation}\label{eqn:expLYM}
    \abs{\bd_n\A} \ge \left(\frac{r}{s+1}+\frac{s}{r+1}\right)\abs{\A} + \sqrt{\frac{n}{rs}}\alpha(1-\alpha)\ncr.
  \end{equation}
\end{theorem}

\begin{remark}
  This phenomenon is reminiscent in the sum-product theorem of Bourgain, Katz and Tao~\cite{bourgain2004sum}. Given a subset $A$ of a finite field $\F_p$, the sum set $A + A = \dset{a + b}{a,b \in A}$ could have size comparable to $A$ if $A$ behaves like an arithmetic progression, and the product set $A \cdot A = \dset{a\cdot b}{a, b \in A}$ could have size comparable to $A$ if $A$ behaves like a geometric progression. However, the sum-product theorem indicates that a ``non-trivial'' $A$ cannot simultaneously behave like an arithmetic progression and a geometric progression.
\end{remark}

Our proof of \cref{thm:expLYM}, given in \cref{sec:local}, is inspired by the work of Christofides, Ellis and Keevash~\cite{keevash}. They established a vertex isoperimetric inequality for the graph $S_n(r)$ with the vertex set $\cS_n(r)$, where two subsets are adjacent if their symmetric difference has size two. Their inequality is an approximate version of a folklore conjecture \cite[Conjecture~1]{bollobas2004isoperimetric} reported by Bollob\'as and Leader. Using a construction in \cite{keevash}, we show that \cref{thm:expLYM} is sharp for $r, s \ge \eps n$ up to a constant factor depending only on $\eps$ and $\alpha$.

\begin{proposition} \label{prop:local-expansion-sharp}
  For every $\eps \in (0,1/2)$ and $n, r, s \in \N$ and $r, s \ge r_0$ such that $n = r+s$ and $r,s\ge \eps n$, and for every $\alpha \in [0,1]$, there exists $\C \subseteq \cS_n(r)$ of size $\floor{\alpha\ncr}$ such that
  \[
    \abs{\bd_n\C} \le \left(\frac{r}{s+1} + \frac{s}{r+1}\right)\abs{\C} + O_\eps\left( 1/\sqrt{n} \right)\ncr.
  \]
\end{proposition}

\paragraph{Acknowledgement.} We are grateful to Emanuel Milman for helpful discussions.

\section{Isoperimetric inequality for Hamming balls} \label{sec:ball}

We need the following simple estimate of $\abs{\cS_n(r)}$ in terms of $\abs{\B_n(r)}$. We postpone its proof to \cref{sec:binomial}

\begin{lemma} \label{lem:monotone}
  For every $0 \le r < n$,
  \[
    \frac{\abs{\cS_n(r)}}{\abs{\B_n(r)}} \ge \frac{\abs{\cS_n(r+1)}}{\abs{\B_n(r+1)}}.
  \] If in addition $n \ge 3$ and $r \le n/2$, then
  \[
    \abs{\cS_n(r)} \ge \abs{\B_n(r)} / \sqrt n.
  \]
\end{lemma}

The next technical lemma readily gives a lower bound on the vertex boundary in Hamming balls.

\begin{lemma} \label{thm:main_bound}
  For every $n, R \in \N$ such that $R \le n$, and every nonempty $\A \subseteq \B_n(R)$, set
  \[
    \eps := \frac{R}{2n},
    \quad
    r_0 := \min\dset{r \le R}{\abs{\B_n(r)} \ge \eps \abs{\A}},
    \quad
    c := 1 - \frac{1}{\abs{\B_n(R)}/\abs{\A}-\eps}.
  \]
  If $n \ge 80$ and $R \le n - r_0$, then the vertex boundary of $\A$ in the Hamming ball $B_n(R)$ satisfies
  \[
    \abs{\bd_{B_n(R)}\A} \ge \frac{2c\sqrt{r_0}}{5n}\eps\abs{\A}.
  \]
\end{lemma}

\begin{proof}
  We may assume that $r_0 \ge 1$ and $c > 0$; otherwise the vertex isoperimetric inequality would become trivial. Because $r_0 \le R \le n - r_0$, we know that $R \ge 1$ and
  \begin{equation} \label{eqn:r0-n/2}
    r_0 \le n/2.
  \end{equation}
  By our choice of $r_0$, we get
  \begin{equation} \label{eqn:bn-r0-1}
    \abs{\B_n(r_0 - 1)} < \eps\abs{\A}.
  \end{equation}

  Our goal is prove
  \begin{equation} \label{eqn:c6ro}
    \sum_{r=0}^R b_r \ge \frac{2c\sqrt{r_0}}{5n}\eps\abs{\A},
  \end{equation}
  where
  \[
    b_r := \abs{\bd_n\A \cap \cS_n(r)} \quad \text{ for }0 \le r \le R.
  \]
  We shall analyze the distribution of $\A$ under the partition $\B_n(R) = \bigcup_{r=0}^R S_n(r)$. To that end, we set
  $\A_r := \A \cap \cS_n(r)$ for $0 \le r \le R$.

  \begin{myclaim} \label{claim:large-top-layer}
    If $\abs{\A_R} \ge (1-1.94\eps)\abs{\A}$, then \eqref{eqn:c6ro} holds.
  \end{myclaim}

  \begin{claimproof}[Proof of \cref{claim:large-top-layer}]
    From \cref{lem:LYM}, we know that
    \[
      \abs{\bm_n\A_R} \ge \frac{R}{n-R+1}\abs{\A_R},
    \]
    which implies that
    \begin{multline*}
      b_{R-1} \ge \abs{\bm_n\A_R} - \abs{\A_{R-1}}
      \ge \frac{R}{n-R+1}\abs{\A_R} - (\abs{\A} - \abs{\A_R}) \\
      = \frac{n+1}{n-R+1}\abs{\A_R} - \abs{\A}
      \ge \left(\frac{n+1}{n-R+1}\left(1-1.94\eps\right)-1\right)\abs{\A}.
    \end{multline*}
    Using $\eps = R / (2n)$ and the assumptions that $n \ge 3$ and $R \ge 1$, we can simplify the coefficient of $\abs{\A}$ above as follows:
    \[
      \frac{n+1}{n-R+1}\left(1-1.94\eps\right)-1 = \frac{2n - 1.94(n+1)}{n-R+1}\eps = \frac{0.06n-1.94}{n-R+1}\eps.
    \]
    Because $n \ge 80$, one can check that
    \[
      \frac{0.06n-1.94}{n-R+1} \ge \frac{0.06n - 1.94}{n} \ge \frac{\sqrt{2n}}{5n},
    \]
    which implies that
    \begin{equation*}
      \sum_{r=0}^R b_r \ge b_{R-1} \ge \frac{\sqrt{2n}}{5n}\eps\abs{\A} \stackrel{\eqref{eqn:r0-n/2}}{\ge} \frac{2\sqrt{r_0}}{5n}\eps\abs{\A}. \qedhere
    \end{equation*}
  \end{claimproof}

  Because of \cref{claim:large-top-layer}, hereafter we only consider the case that
  \begin{equation} \label{eqn:ar_assumption}
    \abs{\A_R} \le (1-1.94\eps)\abs{\A} .
  \end{equation}
  In particular \eqref{eqn:ar_assumption} implies that $\abs{\B_n(R-1)} \ge \abs{\A} - \abs{\A_R} \ge \eps\abs{\A}$, and so
  \[
    1 \le r_0 \le R-1.
  \]

  \begin{myclaim} \label{claim:2c3}
    At least one of the following holds:
    \begin{subequations}
      \begin{gather}
        \sum_{r=r_0}^{R-1}\abs{\alpha_r-\alpha_{r+1}} \ge 4c/7, \label{eqn:imp_1} \\
        \alpha_{r} \le 1 - 3c/7, \quad\text{for }r_0 \le r \le R-1, \label{eqn:imp_2}
      \end{gather}
    \end{subequations}
    where the density of $\A_r$ is defined by $\alpha_r := \abs{\A_r}/\abs{\cS_n(r)}$.
  \end{myclaim}

  \begin{claimproof}[Proof of \cref{claim:2c3}]
    For the sake of contradiction assume that neither \eqref{eqn:imp_1} nor \eqref{eqn:imp_2} holds.
    The negation of \eqref{eqn:imp_1} implies that $\alpha_{r'} - \alpha_r < 4c/7$ for all $r_0 \le r, r' \le R$. The negation of \eqref{eqn:imp_2} means that $\alpha_{r'} > 1-3c/7$ for some $r_0 \le r' \le R-1$. Therefore
    \[
      \alpha_r > \alpha_{r'} - 4c/7 > 1-c,\quad\text{for } r_0 \le r \le R,
    \]
    which implies that
    \[
      \abs{\A} \ge \sum_{r=r_0}^{R}\alpha_r\abs{\cS_n(r)} > (1-c)\left(\abs{\B_n(R)}-\abs{\B_n(r_0-1)}\right)
      \stackrel{\eqref{eqn:bn-r0-1}}{\implies}
      1 - c < \frac{1}{\abs{\B_n(R)}/\abs{\A}-\eps}
    \]
    contradicting the definition of $c$.
  \end{claimproof}

  The proof proceeds by analyzing two different scenarios arising from \eqref{eqn:imp_1} and \eqref{eqn:imp_2} --- the former deals with sets $\A_r$
  whose densities are not equally distributed, whereas the latter deals with sets $\A_r$ whose densities are not very close to $1$.

  \paragraph{Case 1.} Suppose \eqref{eqn:imp_1} holds. For every $r_0 \le r \le R$, since $R \le n - r_0$, we know that $\abs{\cS_n(r)} \ge \abs{\cS_n(r_0)}$. Since $r_0 \le R \le n - r_0$, and in particular $r_0 \le n/2$, and the assumption that $n \ge 3$, \cref{lem:monotone} gives $\abs{\cS_n(r_0)} \ge \abs{\B_n(r_0)}/\sqrt{n}$. Because $\abs{\B_n(r_0)} \ge \eps\abs{\A}$, we know that
  \[
    \abs{\cS_n(r)} \ge \eps\abs{\A}/\sqrt{n} \stackrel{\eqref{eqn:r0-n/2}}{\ge} \frac{\sqrt{2r_0}}{n}\eps\abs{\A} \ge \frac{7\sqrt{r_0}}{5n}\eps\abs{\A}, \quad\text{for }r_0 \le r \le R.
  \]
  By \cref{lem:LYM}, for every $0 \le r \le R-1$, we have
  \begin{align*}
    b_r & \ge \abs{\bm_n\A_{r+1}} - \abs{\A_r} \ge \frac{r+1}{n-r}\abs{\A_{r+1}} - \abs{\A_r} = (\alpha_{r+1} - \alpha_r)\abs{\cS_n(r)},\\
    b_{r+1} & \ge \abs{\bp_n\A_r} - \abs{\A_{r+1}} \ge \frac{n-r}{r+1}\abs{\A_r}- \abs{\A_{r+1}} = (\alpha_r - \alpha_{r+1})\abs{\cS_n(r+1)}.
  \end{align*}
  Combining the last three inequalities, for every $r_0 \le r \le R-1$, we obtain
  \[
    \max(b_r, b_{r+1}) \ge \abs{\alpha_r - \alpha_{r+1}}\frac{7\sqrt{r_0}}{5n}\eps\abs{\A}.
  \]
  Summing over $r$ implies \eqref{eqn:c6ro}:
  \begin{equation*}
    \sum_{r=0}^R b_r \ge \frac12\sum_{r=r_0}^{R-1}\max(b_r,b_{r+1}) \ge \frac12\left(\sum_{r=r_0}^{R-1} \abs{\alpha_r-\alpha_{r+1}}\right)\frac{7\sqrt{r_0}}{5n}\eps\abs{\A} \stackrel{\eqref{eqn:imp_1}}{\ge} \frac{2c\sqrt{r_0}}{5n}\eps\abs{\A}.
  \end{equation*}

  \paragraph{Case 2.} Suppose \eqref{eqn:imp_2} holds.
  By \cref{lem:LYM}, we know that
  \begin{align*}
    \delta_r^+ := \abs{\bp_n\A_r} - \frac{n-r}{r+1}\abs{\A_r} \ge 0, & \quad\text{for }0 \le r \le R-1\\
    \delta_r^- := \abs{\bm_n\A_r} - \frac{r}{n-r+1}\abs{\A_r} \ge 0, & \quad\text{for }1 \le r \le R.
  \end{align*}
  Using $\delta_r^+$ and $\delta_r^-$, we can estimate $b_r$ and $b_{r+1}$ more precisely:
  \begin{align*}
    b_r & \ge \abs{\bm_n\A_{r+1}} - \abs{\A_r} = \frac{r+1}{n-r}\abs{\A_{r+1}} + \delta_{r+1}^- - \abs{\A_r},\\
    b_{r+1} & \ge \abs{\bp_n\A_r} - \abs{\A_{r+1}} = \frac{n-r}{r+1}\abs{\A_{r}} + \delta_r^+ - \abs{\A_{r+1}},
  \end{align*}
  which implies for $0 \le r \le R-1$ that
  \[
    \frac{n-r}{n}b_r + \frac{r+1}{n}b_{r+1} \ge \frac{n-r}{n}\delta_{r+1}^- + \frac{r+1}{n}\delta_r^+.
  \]
  Summing over $r_0 - 1 \le r \le R-1$, we obtain
  \begin{multline*}
    \sum_{r=0}^R b_r \ge \sum_{r=r_0-1}^{R-1}\frac{n-r}{n}b_r + \sum_{r=r_0}^R\frac{r}{n}b_r = \sum_{r=r_0-1}^{R-1}\frac{n-r}{n}b_r + \frac{r+1}{n}b_{r+1} \\
    \ge \sum_{r=r_0-1}^{R-1}\frac{n-r}{n}\delta_{r+1}^- + \frac{r+1}{n}\delta_r^+ \ge \sum_{r=r_0}^{R-1}\frac{n-r+1}{n}\delta_r^- + \frac{r+1}{n}\delta_r^+.
  \end{multline*}
  From \cref{thm:expLYM}, we know that
  \[
    \delta_r^- + \delta_r^+ \ge \sqrt{\frac{n}{r(n-r)}}(1-\alpha_r)\abs{\A_r} \stackrel{\eqref{eqn:imp_2}}{\ge} \frac{3c}{7}\sqrt{\frac{n}{r(n-r)}}\abs{\A_r}.
  \]
  For $r_0 \le r \le R - 1$, because $R \le n - r_0$, we obtain
  \begin{multline*}
    \frac{n-r+1}{n}\delta_r^- + \frac{r+1}{n}\delta_r^+ \ge \min\left( \frac{n-r}{n},\frac{r}{n} \right)(\delta_r^- + \delta_r^+) \\
    \ge \frac{3c}{7} \min\left( \sqrt{\frac{n-r}{rn}},\sqrt{\frac{r}{n(n-r)}} \right)\abs{\A_r} \ge \frac{3c\sqrt{r_0}}{7n}\abs{\A_r}.
  \end{multline*}
  Therefore we obtain
  \[
    \sum_{r=0}^R b_r \ge \frac{3c\sqrt{r_0}}{7n}\sum_{r=r_0}^{R-1}\abs{\A_r},
  \]
  which implies \eqref{eqn:c6ro} through the following fact:
  \begin{equation*}
    \sum_{r=r_0}^{R-1}\abs{\A_r} \ge \abs{\A} - \abs{\A_R} - \abs{\B_n(r_0-1)} \stackrel{(\ref{eqn:bn-r0-1},\ref{eqn:ar_assumption})}{\ge} 0.94\eps\abs{\A}. \qedhere
  \end{equation*}
\end{proof}

\begin{proof}[Proof of \cref{thm:iso-ball}]
  Suppose $\rho \in (0,1/2)$, $R \le n/2$ and $\A \subseteq \B_n(R)$ such that
  \begin{equation} \label{eqn:rho-a-assumption}
    \abs{\B_n(\floor{\rho n})} \le \abs{\A} \le \abs{\B_n(R)} - \abs{\B_n(\floor{\rho n})}.
  \end{equation}
  We break the proof into two cases.

  \paragraph{Case 1.} Suppose $\abs{\A} \le \abs{\B_n(R)}/2$. We would like to apply \cref{thm:main_bound} to $\A$. Recall the definitions of $\eps$, $r_0$ and $c$ in \cref{thm:main_bound}:
  \[
    \eps := R/(2n), \quad r_0 := \min\dset{r \le R}{\abs{\B_n(r)} \ge \eps \abs{\A}}, \quad c := 1 - \frac{1}{\abs{\B_n(R)}/\abs{\A}-\eps}.
  \]
  Because $R \le n/2$, we have $\eps \in (0,1/4)$, and so
  \[
    c \ge 1 - \frac{1}{2 - 1/4} = \frac37.
  \]
  Moreover, the assumption \eqref{eqn:rho-a-assumption} implies that $\abs{\B_n(R)} \ge 2\abs{\B_n(\floor{\rho n})}$. Hence $R > \rho n$ and
  \[
    \eps > \rho/3.
  \]

  Let $r_1$ be a positive integer to be chosen later. By \cref{lem:monotone}, we have
  \[
    \eps\abs{\A} \ge \frac{\rho}{3}\abs{\B_n(\floor{\rho n})}
    \ge \frac{\rho}{3}\left(\frac{\abs{\cS_n(\floor{\rho n})}}{\abs{\cS_n(\floor{\rho n} - 1)}}\right)^{r_1}\abs{\B_n(\floor{\rho n}  - r_1)}
    \ge \frac{\rho}{3}\left(\frac{1-\rho}{\rho}\right)^{r_1}\abs{\B_n(\floor{\rho n}  - r_1)}.
  \]
  Because $\rho \in (0,1/2)$, for some $r_1$ depending only on $\rho$, we have
  \[
    \eps\abs{\A} \ge \abs{\B_n(\floor{\rho n}  - r_1)}.
  \]
  Therefore $r_0 \ge \floor{\rho n} - r_1$. For $n \ge n_0$, where $n_0$ depends only on $\rho$, \cref{thm:main_bound} yields
  \[
    \abs{\bd_{B_n(R)}\A} \ge \frac{2c}{5}\frac{\sqrt{r_0}}{n}\eps\abs{\A} \ge \frac{6}{35}\frac{\sqrt{\floor{\rho n} - r_1}}{n}\frac{\rho}{3}\abs{\A} \ge \frac{\rho^{3/2}}{18\sqrt{n}}\abs{\A}.
  \]
  
  \paragraph{Case 2.} Suppose $\abs{\A} > \abs{\B_n(R)}/2$. Set
  \[
    \A^c := \B_n(R)\setminus\A \quad\text{and}\quad \A' := \A^c\setminus \bd_n\A.
  \]
  We would like to apply \cref{thm:main_bound} to $\A'$. The parameters in \cref{thm:main_bound} are
  \[
    \eps := R / (2n), \quad
    r_0' := \min\dset{r \le R}{\abs{\B_n(r)} \ge \eps\abs{\A'}}, \quad
    c' := 1-\frac{1}{{\abs{\B_n(R)}}/{\abs{\A'}}-\eps}.
  \]
  Note that $\abs{\A'} \le \abs{\A^c} \le \abs{\B_n(R)}/2$. In this case,
  \[
    c' \ge 1 - \frac{1}{2-1/4} = \frac{3}{7}.
  \]
  We may assume that $\abs{\bd_{B_n(R)}\A} \le \abs{\A^c}/2$, because otherwise we are done. Thus 
  \[
    \abs{\A'} = \abs{\A^c} - \abs{\bd_{B_n(R)}\A} \ge \frac12\abs{\A^c} = 
    \frac12(\abs{\B_n(R)} - \abs{\A}) \ge \frac12\abs{\B_n(\floor{\rho n})}.
  \]
  Similarly to Case~1, $r_0' \ge \floor{\rho n}-r_2$ for some $r_2$ depending only on $\rho$. \cref{thm:main_bound} yields
  \[
    \abs{\bd_{B_n(R)}\A'} \ge \frac{2c'}{5}\frac{\sqrt{r_0'}}{n}\eps\abs{\A'}
    \ge \frac{6}{35}\frac{\sqrt{\floor{\rho n}-r_2}}{n}\frac{\rho}{3}\abs{\A'} 
    = \frac{2}{35}\frac{\rho^{3/2}-O(1/n)}{\sqrt{n}}\left(\abs{\A^c}-\abs{\bd_{B_n(R)}\A}\right).
  \]
  Observe that $\bd_{B_n(R)}\A' \subseteq \bd_{B_n(R)}\A$.
  Indeed, if $v \in \bd_{B_n(R)}\A'$ then
  $v \not \in \A' \cup \A$ which implies $v \in \bd_{B_n(R)}\A$. Thus
  \[
    \abs{\bd_{B_n(R)}\A}
    \ge \frac{2}{35}\frac{\rho^{3/2}-O(1/n)}{\sqrt{n}}\left(\abs{\A^c}-\abs{\bd_{B_n(R)}\A}\right).
  \]
  For $n \ge n_0$, where $n_0$ depends only on $\rho$, we can rewrite the above
  \begin{equation*}
    \abs{\bd_{B_n(R)}\A} \ge \frac{1}{1+O(1/\sqrt{n})}\frac{2}{35}\frac{\rho^{3/2}-O(1/n)}{\sqrt{n}}\abs{\A^c} \ge \frac{\rho^{3/2}}{18\sqrt{n}}\abs{\A^c}. \qedhere
  \end{equation*}

\end{proof}

\section{Local expansion estimate} \label{sec:local}

Our proof of \cref{thm:expLYM} is by induction, and its outline is similar to the proof in~\cite{keevash}. However ours differs in one key aspect --- we need to choose ``where to apply induction'', whereas in~\cite{keevash} this was immaterial. Besides there are several other technical difficulties we need to overcome.

We shall utilize the following criterion for two interlacing real-rooted quadratic polynomials.

\begin{proposition} \label{prop:interlace}
  Let $p_1(x) = x^2 + B_1x + C_1$ and $p_2(x) = x^2 + B_2x + C_2$ be two monic quadratic polynomials with real coefficients. Suppose $p_i(x)$ has two distinct real roots $x_i^- < x_i^+$ for $i \in \sset{1,2}$. If $x_1^- < x_2^-$, $x_1^+ < x_2^+$ and $(C_1 - C_2)^2 + (B_1 - B_2)(B_1C_2 - B_2C_1) < 0$, then $x_2^- < x_1^+$.
\end{proposition}

\begin{proof}
  Notice that $p_1(x) = p_2(x)$ at $x = x_0 := -(C_1-C_2)/(B_1-B_2)$. Since $p_1(x_0) = p_2(x_0)$ and
  \[
    p_1(x_0) = \left( \frac{C_1-C_2}{B_1-B_2} \right)^2 - B_1\left( \frac{C_1-C_2}{B_1-B_2} \right) + C_1 = \frac{(C_1-C_2)^2+(B_1-B_2)(B_1C_2-B_2C_1)}{(B_1-B_2)^2} < 0,
  \]
  we know that $x_2^- < x_0 < x_1^+$.
\end{proof}

\begin{proof}[Proof of \cref{thm:expLYM}]
  Without loss of generality, we assume
  \[
    \alpha := \abs{\A} / \abs{\B_n(r)} \in (0,1);
  \]
  because \eqref{eqn:expLYM} would follow from \cref{lem:LYM} immediately when $\alpha \in \sset{0,1}$. We may also assume that $r \le s$, since \cref{thm:expLYM} is symmetric with respect to $r$ and $s$. Indeed, if we replace $\A\subseteq\cS_n(r)$ by $\A' = \dset{[n] \setminus X}{X \in \A}\subseteq\cS_n(n-r)$, then $\abs{\A} = \abs{\A'}$ and $\abs{\bd_n\A} = \abs{\bd_n\A'}$, while the right hand side of \eqref{eqn:expLYM} is invariant under this replacement.
 
  For the $r = 1$ base case, we know $\bm_n\A$ and $\bp_n\A$ precisely:
  \[
    \bm_n\A = \cS_n(0)
    \quad\text{and}\quad
    \bp_n\A = \dset{X \in \cS_n(2)}{X\cap(\cup\A)\neq \varnothing}.
  \]
  Estimate $\abs{\bd_n\A}$ as follows:
  \begin{multline*}
    \abs{\bd_n\A} - \left(\frac{s}{r+1}+\frac{r}{s+1}\right)\abs{\A}
    = \binom{n}{0} + \binom{n}{2} - \binom{n-\alpha n}{2} - \left(\frac{n-1}{2}+\frac{1}{n}\right)\alpha n \\
    = (1-\alpha)\left(\frac{\alpha}{2}n^2+1\right)
    \ge (1-\alpha) 2\sqrt{\frac{\alpha}{2}} n \ge \sqrt{2}(1-\alpha)\alpha n
    \ge \sqrt{\frac{n}{n-1}}\alpha(1-\alpha)n.
  \end{multline*}
  
  For the inductive step, let $r \ge 2$. We first choose where to apply induction. Since each set in $\A$ has size $r$, by the pigeonhole principle, some element of $[n]$ appears in at least $\frac{r}{n}\abs{\A} = \frac{r}{n}\cdot\alpha\ncr = \alpha\binom{n-1}{r-1}$ sets of $\A$. Without loss of generality, we may assume that $n$ is this element. Decompose the projection of $\A$ onto $[n-1]$ into two families:
  \[
    \A_0 := \dset{X \subseteq [n-1]}{X \in \A}
    \quad\text{and}\quad
    \A_1 := \dset{X \subseteq [n-1]}{X \cup \sset{n} \in \A}.
  \]
  Thus, $\A_0 \subseteq \cS_{n-1}(r)$, $\A_1 \subseteq \cS_{n-1}(r-1)$, and $\abs{\A_1} \ge \alpha\binom{n-1}{r-1}$. We set some notation.

  \paragraph{Notation.}
  Set
  \[
    \alpha_0 := \abs{\A_0}/\binom{n-1}{r}
    \quad\text{and}\quad
    \alpha_1 := \abs{\A_1}/\binom{n-1}{r-1}.
  \]
  As $\abs{\A} = \abs{\A_0} + \abs{\A_1}$, we have $\alpha\ncr = \alpha_0\binom{n-1}{r} + \alpha_1\binom{n-1}{r-1}$,
  which implies
  \begin{equation} \label{eqn:convex}
    \alpha = \frac{s}{n}\alpha_0 + \frac{r}{n}\alpha_1.
  \end{equation}
  Since $\abs{\A_1} \ge \alpha\binom{n-1}{r-1}$, we know that $\alpha_1 \ge \alpha$, and hence $\alpha_0 \le \alpha \le \alpha_1$.
  Set
  \begin{equation} \label{eqn:def-x}
    x := \alpha_1 - \alpha_0.
  \end{equation}
  Because $0 \le \alpha_0 \le \alpha_1 \le 1$, we know that
  \[
    0 \le x \le x^*, \quad\text{where }x^* := \frac{n}{s}(1-\alpha).
  \]
  The following constants arise from the induction hypothesis.
  \[
    c := \sqrt{\frac{n}{rs}}, \quad c_0 := \sqrt{\frac{n-1}{r(s-1)}}, \quad c_1 := \sqrt{\frac{n-1}{(r-1)s}}, \quad t := \frac{s}{r+1} - \frac{r}{s+1} \ge 0.
  \]
  It is easy to check:
  \begin{equation} \label{eqn:aux_a}
    c < c_0 \le c_1.
  \end{equation}

  \paragraph{Two estimations.}
  By the induction hypothesis, we estimate the vertex boundary of $\A_0$:
  \begin{equation} \label{eqn:ind0}
    \begin{split}
      \frac{\abs{\bd_{n-1}\A_0}}{\ncr}
      & \ge \left(\frac{s-1}{r+1} + \frac{r}{s}\right)\frac{\abs{\A_0}}{\ncr} + c_0\alpha_0(1-\alpha_0)\frac{\binom{n-1}{r}}{\ncr} \\
      & = \left(\frac{s-1}{r+1} + \frac{r}{s}\right){\alpha_0}\frac{s}{n} + c_0\alpha_0(1-\alpha_0)\frac{s}{n} \\
      & = \left(\frac{s}{r+1} + \frac{r}{s+1}\right)\frac{s}{n}\alpha_0 - \frac{t}{n}\alpha_0 + c_0\alpha_0(1-\alpha_0)\frac{s}{n}.
    \end{split}
  \end{equation}
  Similarly, we estimate the vertex boundary of $\A_1$:
  \begin{equation} \label{eqn:ind1}
    \frac{\abs{\bd_{n-1}\A_1}}{\ncr} \ge \left(\frac{s}{r+1}+\frac{r}{s+1}\right)\frac{r}{n}\alpha_1 + \frac{t}{n}\alpha_1+c_1\alpha_1(1-\alpha_1)\frac{r}{n}.
  \end{equation}

  Now, we can bound $\abs{\bd_n\A}$ from below in two ways:
  \begin{subequations}
    \begin{align}
      \abs{\bd_n\A} & \ge \abs{\bd_{n-1}\A_0} + \abs{\bd_{n-1}\A_1}, \label{eqn:b1} \\
      \abs{\bd_n\A} & \ge \abs{\bd_{n-1}\A_1} + \abs{\bp_{n-1}\A_0} + \abs{\A_1}. \label{eqn:b2}
    \end{align}    
  \end{subequations}
  On the one hand, \eqref{eqn:b1} holds because
  \[
    \bd_{n-1}\A_0 \subseteq \dset{X \in \bd_n\A}{n \not\in X}
    \quad\text{and}\quad
    \dset{X \cup \sset{n}}{X \in \bd_{n-1}\A_1} \subseteq \dset{X \in \bd_n\A}{n \in X}.
  \]
  On the other hand, \eqref{eqn:b2} holds because
  \begin{gather*}
    \dset{X \cup \sset{n}}{X \in \bd_{n-1}\A_1} \subseteq \dset{X \in \bd_n\A}{n \in X}, \\
    \bp_{n-1}\A_0 \subseteq \dset{X \in \bp_n\A}{n \not\in X},
    \quad\text{and}\quad
    \A_1 \subseteq \dset{X \in \bm_n\A}{n \not\in X}.
  \end{gather*}
  Combining \eqref{eqn:b1}, \eqref{eqn:ind0} and \eqref{eqn:ind1}, we obtain the first estimation:
  \begin{multline} \label{eqn:l1}
    \frac{\abs{\bd_n\A}}{\ncr} \ge \left(\frac{s}{r+1}+\frac{r}{s+1}\right)\left(\frac{s}{n}\alpha_0+\frac{r}{n}\alpha_1\right) + \frac{t}{n}(\alpha_1-\alpha_0) +c_0\alpha_0(1-\alpha_0)\frac{s}{n}+c_1\alpha_1(1-\alpha_1)\frac{r}{n} \\
    \stackrel{(\ref{eqn:convex},\ref{eqn:def-x})}{=} \left(\frac{s}{r+1}+\frac{r}{s+1}\right)\alpha + \frac{1}{n}\bigg[tx + c_0\alpha_0(1-\alpha_0) s + c_1\alpha_1(1-\alpha_1)r\bigg].
  \end{multline}
  From \cref{lem:LYM}, we get
  \begin{equation*}
    \abs{\bp_{n-1}\A_0} \ge \frac{s-1}{r+1}\abs{\A_0} = \frac{s-1}{r+1}\frac{s}{n}\alpha_0\ncr,
  \end{equation*}
  which together with \eqref{eqn:b2} and \eqref{eqn:ind1} yields the second estimation:
  \begin{multline} \label{eqn:l2}
    \frac{\abs{\bd_n\A}}{\ncr}
    \ge \left(\frac{s}{r+1}+\frac{r}{s+1}\right)\frac{r}{n}\alpha_1 + \frac{t}{n}\alpha_1+c_1\alpha_1(1-\alpha_1)\frac{r}{n} + \frac{s-1}{r+1}\frac{s}{n}{\alpha_0} + \frac{r}{n}\alpha_1 \\
    = \left(\frac{s}{r+1}+\frac{r}{s+1}\right)\left(\frac{s}{n}\alpha_0+\frac{r}{n}\alpha_1\right) + \frac{t+r}{n}(\alpha_1-\alpha_0) + c_1\alpha_1(1-\alpha_1)\frac{r}{n} \\
    \stackrel{(\ref{eqn:convex},\ref{eqn:def-x})}{=} \left(\frac{s}{r+1}+\frac{r}{s+1}\right)\alpha + \frac{1}{n}\bigg[(t+r)x + c_1\alpha_1(1-\alpha_1)r\bigg].
  \end{multline}

  To simplify notation, denote by $L_1(x)$ and $L_2(x)$ the expressions in the last brackets of \eqref{eqn:l1} and \eqref{eqn:l2} respectively:
  \begin{align*}
    L_1(x) & := tx + sc_0\alpha_0(1-\alpha_0) + rc_1\alpha_1(1-\alpha_1), \\
    L_2(x) & := \left(t+r\right)x + rc_1\alpha_1(1-\alpha_1).
  \end{align*}
  It suffices to show that for all $\alpha \in (0,1)$ and $x \in [0,x^*]$,
  \[
    \max(L_1(x), L_2(x)) \ge nc\alpha(1-\alpha) =: Q.
  \]

  \paragraph{Verification.} Using \eqref{eqn:convex} and \eqref{eqn:def-x}, we can express $\alpha_0, \alpha_1$ in terms of $\alpha$ and $x$:
  \begin{equation} \label{eqn:def-alpha-in-x}
    \alpha_0 = \alpha_0(x) := \alpha - \frac{r}{n}x
    \quad\text{and}\quad
    \alpha_1 = \alpha_1(x) := \alpha + \frac{s}{n}x.
  \end{equation}
  Thus we can view $L_1$ and $L_2$ as quadratic functions of $x$ with coefficients determined by $r, s$ and $\alpha$:
  \begin{align*}
    L_1(x) & = tx + sc_0\left( \alpha - \frac{r}{n}x \right) \left( 1 - \alpha + \frac{r}{n}x \right) + rc_1\left( \alpha + \frac{s}{n}x \right) \left( 1- \alpha - \frac{s}{n}x \right), \\
    L_2(x) & = \left(t+r\right)x + rc_1\left( \alpha + \frac{s}{n}x \right) \left( 1- \alpha - \frac{s}{n}x \right).
  \end{align*}

  We first study the evaluations of $L_1(x)$ at $x = 0$ and $x = x^*$ respectively.
  Observe that
  \[
    L_1(0) = sc_0\alpha(1-\alpha)+rc_1\alpha(1-\alpha) \stackrel{\eqref{eqn:aux_a}}{>} (r+s)c\alpha(1-\alpha) = Q.
  \]
  If $L_1(x^*) \ge Q$, we are done because the leading coefficient of $L_1(x)$ is $-r^2sc_0/n^2-rs^2c_1/n^2$, which is negative, and so $L_1(x) \ge Q$ for $x\in [0, x^*]$. Hereafter we may assume that
  \[
    L_1(x^*) < Q.
  \]
  \setcounter{myclaim}{0}
  \begin{myclaim} \label{claim:alpha}
    If $L_1(x^*) < Q$, then $\alpha < (t+r)/(sc).$
  \end{myclaim}
  \begin{claimproof}[Proof of \cref{claim:alpha}]
    Because $\alpha_0(x^*) = \alpha - \frac{r}{s}(1-\alpha)$ and $\alpha_1(x^*) = 1$, we have
    \[
      L_1(x^*) - Q = t\left( \frac{n}{s}(1-\alpha) \right) + sc_0\left( \alpha - \frac{r}{s}(1-\alpha) \right) \left( 1 - \alpha + \frac{r}{s}(1-\alpha) \right) - nc\alpha(1-\alpha),
    \]
    which after multiplying $s/(n(1-\alpha))$ equals:
    \[
      t + c_0(n\alpha - r) - sc\alpha = (nc_0-sc)\alpha - (rc_0 - t).
    \]
    Because $nc_0 - sc \ge nc - sc = rc > 0$, from the last inequality above, we know that
    \[
      \alpha < \frac{rc_0 - t}{nc_0 - sc}.
    \]
    The claim is implied by the following inequality involving constants determined by $r$ and $s$ only.
    \begin{equation} \label{eqn:check-later-1}
      \frac{rc_0 - t}{nc_0 - sc} < \frac{t+r}{sc}.
    \end{equation}
    We carry out the routine verification of \eqref{eqn:check-later-1} in \cref{sec:inequalities}.
  \end{claimproof}
  
  Next we study the evaluations of $L_2(x)$ and the following variation of $L_2(x)$ at $x = 0$ and $x = x^*$:
  \[
    L_2^-(x) := (t+r)x + rc\alpha_1(1-\alpha_1) = (t+r)x + rc\left( \alpha + \frac{s}{n}x \right) \left( 1- \alpha - \frac{s}{n}x \right).
  \]
  Using the fact that $0 \le \alpha_1(x) \le 1$ for $x\in [0,x^*]$, we observe that
  \[
    L_2(x) \stackrel{\eqref{eqn:aux_a}}{\ge} L_2^-(x) \quad \text{for }x \in [0, x^*].
  \]
  Using the fact that $\alpha_1(x^*) = 1$ and \cref{claim:alpha}, we observe that
  \[
    L_2(x^*) = L_2^-(x^*) = (t+r)x^* = (t+r)\frac{n}{s}(1-\alpha) > nc\alpha(1-\alpha) = Q.
  \]
  Because the leading coefficient of $L_2(x)$ is $-rcs^2/n^2$, which is negative, we may assume that \[
    L_2(0) < Q.
  \]

  \begin{myclaim} \label{claim:roots}
    There exist two roots $x_1^-$ and $x_1^+$ of $L_1(x) = L_2(x)$ such that $x_1^- < 0 < x_1^+ < x^*$, and there exist two roots $x_2^-$ and $x_2^+$ of $L_2^-(x) = Q$ such that $0 < x_2^- < x^* < x_2^+$, and moreover $x_2^- < x_1^+$ (see \cref{fig:quadratic}).
  \end{myclaim}

  \begin{figure}
    \centering
    \begin{tikzpicture}
      \begin{axis} [
        ultra thick,
        axis lines = none,
        domain=-0.1:1.0,
        samples=100,
        legend columns=-1,
        legend style={draw=none, /tikz/every even column/.append style={column sep=12pt}, at={(0.5,-0.2)},anchor=south},
        x post scale=2,
        ymin = -0.5,
      ]
        \addplot [color={rgb,255:red,49; green,114; blue,177}]
        {0.9333 * x + 0.9129 * 4 * (0.4 - 0.3333 * x) * (0.6 + 0.3333 * x) + 1.118 * 2 * (0.4 + 0.6667 * x) * (0.6 - 0.6667 * x)};
        \addplot [color={rgb,255:red,197; green,69; blue,68}]
        {2.9333 * x + 1.118 * 2 * (0.4 + 0.6667 * x) * (0.6 - 0.6667 * x)};
        \addplot [color={rgb,255:red,60; green,139; blue,73}]
        {2.9333 * x + 0.866 * 2 * (0.4 + 0.6667 * x) * (0.6 - 0.6667 * x)};
        \addplot [color={rgb,255:red,96; green,70; blue,164}]
        {1.2471};
        \addplot [color=black]{0};
        \addplot [color=black, dashed] coordinates {(0, 0) (0, 2.7)};
        \node[anchor=north] at (0,-0.07) {$0$};
        \addplot [color=black, dashed] coordinates {(0.2821, 0) (0.2821,1.247)};
        \node[anchor=north] at (0.2821,0) {$x_2^-$};
        \addplot [color=black, dashed] coordinates {(0.3664, 0) (0.3664,1.587)};
        \node[anchor=north] at (0.3664,0) {$x_1^+$};
        \addplot [color=black, dashed] coordinates {(0.9, 0) (0.9, 2.7)};
        \node[anchor=north] at (0.9,-0.04) {$x^*$};
        \legend{$L_1(x)$, $L_2(x)$, $L_2^-(x)$, $y = Q$};
      \end{axis}
    \end{tikzpicture}
    \caption{The graphs of $L_1(x), L_2(x), L_2^-(x)$ and $y = Q$ for $x \in [0, x^*]$, and the intersections for both $L_1(x) = L_2(x)$ and $L^-_2(x) = Q$.} \label{fig:quadratic}
  \end{figure}
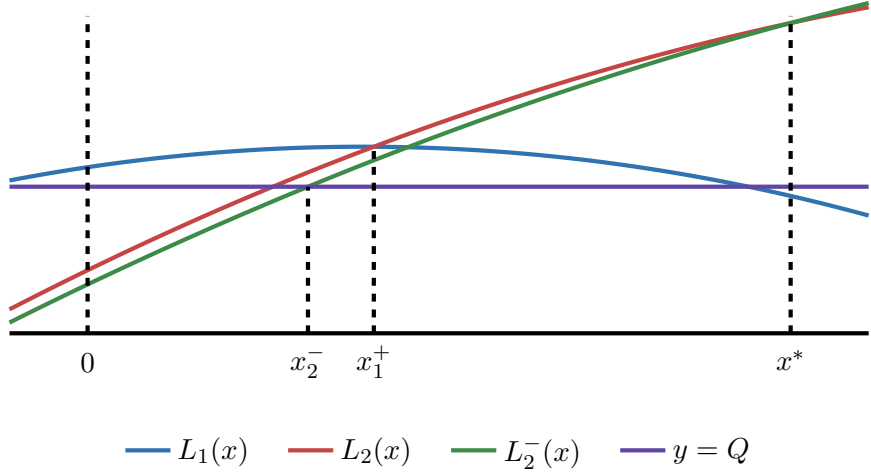

  \begin{claimproof}[Proof of \cref{claim:roots}]
    Note that $L_1(x) - L_2(x)$ is a quadratic polynomial in $x$ with leading coefficient $-r^2s_0/n^2$, which is negative, and moreover $L_1(0) > L_2(0)$ and $L_1(x^*) < L_2(x^*)$. We know that $L_1(x) - L_2(x)$ has two roots $x_1^-$ and $x_1^+$ such that $x_1^- < 0 < x_1^+ < x^*$. Note that the leading coefficient of $L_2^-(x) - Q$ is $-rs^2c/n^2$, which is negative, and moreover $L_2^-(0) < Q < L_2^-(x^*)$. We know that $L_2^-(x) - Q$ has two roots $x_2 ^-$ and $x_2^+$ such that $0 < x_2^- < x^* < x_2^+$.

    We consider the following two monic quadratic polynomials:
    \begin{align*}
      P_1(x) & := -\frac{n^2}{r^2sc_0}\left( L_1(x) - L_2(x) \right) = x^2 + B_1x - C_1, \\
      P_2(x) & := -\frac{n^2}{rs^2c}\left( L_2^-(x) - Q \right) = x^2 - B_2 + C_2,
    \end{align*}
    where
    \begin{gather*}
      B_1 := \frac{n}{r}(1-2\alpha) + \frac{n^2}{rsc_0}, \quad C_1 = \frac{n^2}{r^2}\alpha(1-\alpha), \\
      B_2 := \frac{n}{s}(1-2\alpha) + \frac{n^2(t+r)}{rs^2c}, \quad C_2 = \frac{n^2}{rs}\alpha(1-\alpha).
    \end{gather*}
    To prove $x_2^- < x_1^+$, by \cref{prop:interlace}, it suffices to check
    \[
      (C_1 + C_2)^2 < (B_1 + B_2)(B_2C_1 - B_1C_2),
    \]
    which is equivalent to the following inequalities:
    \[
      \left( \frac{n^2}{r^2} + \frac{n^2}{rs} \right)^2 \alpha^2(1-\alpha)^2 < \left( \frac{n^2}{rs}(1-2\alpha) + \frac{n^2}{rsc_0} + \frac{n^2(t+r)}{rs^2c} \right) \left( \frac{n^4(t+r)}{r^3s^2c} - \frac{n^4}{r^2s^2c_0} \right) \alpha(1-\alpha),
    \]
    which after multiplying both sides by $r^4s^3/(n^6\alpha(1-\alpha))$ is equivalent to:
    \[
      \frac{s}{r}\alpha(1-\alpha) < \left( 1-2\alpha + \frac{t+r}{sc} + \frac{1}{c_0}\right)\left( \frac{t+r}{rc} -\frac{1}{c_0} \right).
    \]
  
    We have successfully eliminated $x$ and reduced the problem to a quadratic inequality of $\alpha$:
    \[
      -\frac{s}{r}\alpha^2 + 2\left( \frac{s}{2r} + \frac{t+r}{rc} - \frac{1}{c_0} \right)\alpha - \left( 1 + \frac{t+r}{sc} + \frac{1}{c_0} \right) \left( \frac{t+r}{rc} - \frac{1}{c_0} \right) < 0,
    \]
    which is ensured if its discriminant is negative. Finally, we note that the negativity of the discriminant is equivalent to
    \begin{equation} \label{eqn:check-later-2}
      \left( \frac{s}{2r} + \frac{t+r}{rc} - \frac{1}{c_0} \right)^2 < \frac{s}{r}\left( 1 + \frac{t+r}{sc} + \frac{1}{c_0} \right) \left( \frac{t+r}{rc} - \frac{1}{c_0} \right).
    \end{equation}
    We carry out the routine verification of \eqref{eqn:check-later-2} in \cref{sec:inequalities}.  
  \end{claimproof}

  Recall that $L_2(x) \ge L_2^-(x)$ for $x\in [0, x^*]$. In particular $L_2(x_2^-) \ge L_2^-(x_2^-) = Q$, which implies that $L_2(x_2^-) \ge Q$ for $x\in [x_2^-, x^*]$ by the concavity of $L_2(x)$. Particularly $L_1(x_1^+) = L_2^-(x_1^+) > Q$, which implies that $L_1(x) > Q$ for $x\in [0, x_1^+]$ by the concavity of $L_1(x)$. Since $x_2^- < x_1^+$ from \cref{claim:roots}, we get the desired inequality $\max(L_1(x), L_2(x)) \ge Q$ for all $x\in [0, x^*]$ for the inductive step.
\end{proof}

\section{Sharpness} \label{sec:sharp}

A random variable $H$ is said to have the \emph{hypergeometric distribution} with parameters $r, m, n$, written as $H \sim \hg(r; m, n)$, if its probability mass function is given by
\[
  \Pr(H = k) = \begin{cases}
    \binom{m}{k}\binom{n-m}{r-k}/\binom{n}{r} & \text{if }k = 0, 1, \dots, r, \\
    0 & \text{otherwise}.
  \end{cases}
\]
We need the following simple fact about hypergeometric distribution. We shall use the inequality $\binom{2m}{m}\ge 2^{2m}/(2\sqrt{m})$ (see, for example, \cite[Proposition 3.6.2]{MR2469243} for a proof).

\begin{proposition} \label{prop:hg-middle}
  If $H \sim \hg(r; \floor{n/2}, n)$, then for all $k \in \N$,
  \[
    \Pr(H = k) \le O\left(\sqrt{\frac{n}{r(n-r)}}\right).
  \]
\end{proposition}

\begin{proof}
  Put $m := \floor{n/2}$. Using $\binom{n}{r} \le \binom{n}{\floor{n/2}} = \Theta(2^n/\sqrt n)$, we compute
  \begin{equation*}
    \Pr(H = k) = \binom{m}{k}\binom{n-m}{r-k}/\binom{n}{r} = \binom{r}{k}\binom{n-r}{m-k}/\binom{n}{m}
    \le O\left(\frac{2^r}{\sqrt{r}}\frac{2^{n-r}}{\sqrt{n-r}}\frac{\sqrt{n}}{2^n}\right). \qedhere
  \end{equation*}
\end{proof}

Now we are ready to prove \cref{prop:iso-ball-sharp,prop:local-expansion-sharp}.

\begin{proof}[Proof of \cref{prop:iso-ball-sharp}]
  Given $\eps \in (0,1/2)$ and $\alpha \in (\eps, 1-\eps)$. Consider $n, R \in \N$ such that $\eps n < R \le n/2$.
  Set $Y := \sset{1,\dots,\floor{n/2}}$, and for all integers $k$, put
  \[
    \C(k) := \dset{X \in \B_n(R)}{\abs{X \cap Y} \le \abs{X}/2 + k}.
  \]
  Because $\C(k) = \varnothing$ for $k < -R/2$, and $\C(k) = \B_n(R)$ for $k > R/2$, we can take $\M$ such that $\C(k-1) \subseteq \M \subseteq \C(k)$, for some integer $k$, and $\abs{\M} = \floor{\alpha\abs{\B_n(R)}}$. Note that
  \[
    \bd_{B_n(R)}\M \subseteq \dset{X \in \B_n(R)}{\abs{X \cap Y} - \floor{\abs{X}/2} \in \sset{k, k+1}}.
  \]
  Thus we estimate the size of $\bd_{B_n(R)}\M$ by
  \[
    \abs{\bd_{B_n(R)}\M} \le \sum_{r=0}^R \Pr(H_r - \floor{r/2} \in \sset{k, k + 1})\binom{n}{r}.
  \]
  By \cref{prop:hg-middle}, we know that, for $R / 2 \le r \le R$,
  \[
    \Pr(H_r - \floor{r/2} \in \sset{k, k + 1}) = O\left(\sqrt{\frac{n}{r(n-r)}}\right) = O_\eps\left(1/\sqrt{n}\right).
  \]
  Thus we further estimate the size of $\bd_{B_n(R)}\M$ by
  \[
    \abs{\bd_{B_n(R)}\M} \le O_\eps\left( 1/\sqrt{n} \right)\abs{\B_n(R)} + \abs{\B_n(R_0)},
  \]
  where $R_0 := \floor{R/2}$. By \cref{lem:monotone}, we know that
  \begin{multline*}
    \frac{\abs{\B_n(R_0)}}{\abs{\B_n(R)}} \le \frac{\abs{\cS_n(R_0)}}{\abs{\cS_n(R)}} = \frac{\binom{n}{R_0}}{\binom{n}{R}} = \frac{R(R-1)\dots(R_0+1)}{(n-R_0)(n-R_0-1)\dots(n-R+1)} \\
    \le \left(\frac{R}{2R-R_0}\right)^{R-R_0} \le \left( \frac23 \right)^{\eps n/2} \le O_\eps\left( \frac{1}{\sqrt{n}} \right).
  \end{multline*}
  Thus $\abs{\bd_{B_n(R)}\M} \le O_\eps\left( 1/\sqrt{n} \right)\abs{\B_n(R)} \le O_\eps\left( 1/\sqrt{n} \right)\min(\abs{\M}, \abs{\B_n(R)\setminus \M})$.
\end{proof}

\begin{proof}[Proof of \cref{prop:local-expansion-sharp}]
  Given $\eps \in (0, 1/2)$ and $\alpha \in [0,1]$, consider $n, r, s \in \N$ such that 
  \[
    r + s = n \quad\text{and}\quad \eps n \le r, s \le (1-\eps)n.
  \]
  Set $Y := \sset{1,\dots,\floor{n/2}}$ and for all integers $k$,
  \[
    \C(k) := \dset{X \in \cS_n(r)}{\abs{X \cap Y} \le r/2 + k}.
  \]
  Because $\C(k) = \varnothing$ for $k < -r/2$, and $\C(k) = \S_n(r)$ for $k > r/2$, we can take $\C$ such that $\C(k-1) \subseteq \C \subseteq \C(k)$, for some integer $k$, and $\abs{\C} = \floor{\alpha\abs{\cS_n(r)}}$. Set
  \[
    \C^+ := \dset{X \in \cS_n(r+1)}{\abs{X \cap Y} \le r / 2 + k - 1}.
  \]
  Because $\bm_n\C^+ \subseteq \C(k-1) \subseteq \C$, \cref{lem:LYM} gives that
  \[
    \abs{\C^+} \le \frac{s}{r+1}\abs{\bm_n\C^+} \le \frac{s}{r+1}\abs{\C}.
  \]
  Note that
  \[
    \bp_n\C \setminus \C^+ \subseteq \dset{X \in \cS_n(r+1)}{\abs{X \cap Y} - \floor{r/2} \in \sset{k, k+1}}.
  \]
  The right hand side of the above has size
  \[
    \Pr(H_{r+1} \in \sset{\floor{r/2} + k, \floor{r/2} + k+1})\binom{n}{r+1}, \quad\text{where }H_{r+1} \sim \hg(r+1; \floor{n/2}, n).
  \]
  Thus by \cref{prop:hg-middle}, we can estimate $\abs{\bp_n\C}$ as follows:
  \[
    \abs{\bp_n\C} \le \abs{\C^+} + O\left(\sqrt{\frac{n}{(r+1)(s-1)}}\right)\binom{n}{r+1} \le \frac{s}{r+1}\abs{\C} + O_\eps\left( 1/\sqrt{n} \right)\ncr.
  \]
  The lower shadow of $\C$ can be estimated similarly:
  \begin{equation*}
    \abs{\bm_n\C} \le \frac{r}{s+1}\abs{\C} + O_\eps\left( 1/\sqrt{n} \right)\ncr. \qedhere
  \end{equation*}
\end{proof}

\bibliographystyle{alpha}
\bibliography{iso_ball}

\appendix

\section{Proof of \texorpdfstring{\cref{lem:monotone}}{Lemma~6}} \label{sec:binomial}

\begin{proof}[Proof of \cref{lem:monotone}]
  Observe that for every $k \le r$, we have
  \[
    \frac{\binom{n}{k}}{\ncr} = \frac{k+1}{r+1} \cdot \frac{n-r}{n-k} \cdot \frac{\binom{n}{k+1}}{\binom{n}{r+1}} \le \frac{\binom{n}{k+1}}{\binom{n}{r+1}},
  \]
  which implies that
  \[
    \frac{\abs{\B_n(r)}}{\abs{\cS_n(r)}}
    = \sum_{k=0}^r\frac{\binom{n}{k}}{\ncr}
    \le \sum_{k=0}^r\frac{\binom{n}{k+1}}{\binom{n}{r+1}}
    \le \sum_{k=0}^{r+1}\frac{\binom{n}{k}}{\binom{n}{r+1}}
    = \frac{\abs{\B_n(r+1)}}{\abs{\cS_n(r+1)}}.
  \]
  In particular, when $r \le n/2$,
  \[
    \frac{\abs{\cS_n(r)}}{\abs{\B_n(r)}}
    \ge \frac{\abs{\cS_n(\floor{n/2})}}{\abs{\B_n(\floor{n/2})}}.
  \]
  \paragraph{Case: $n$ is even.} We know that 
  \[
    \abs{\cS_n(\floor{n/2})} = \binom{n}{n/2} \quad\text{and}\quad \abs{\B_n(\floor{n/2})} = 2^{n-1} + \frac{1}{2}\binom{n}{n/2}.
  \]
  Since $\binom{n}{n/2} \ge 2^n/(\sqrt{2n})$ and $n \ge 3$, we get
  \[
    \frac{\abs{\cS_n(\floor{n/2})}}{\abs{\B_n(\floor{n/2})}} = \frac{1}{2^{n-1}/\binom{n}{n/2}+1/2} \ge \frac{1}{\sqrt{2n}/2+1/2} \ge \frac{1}{\sqrt{n}}.
  \]
  
  \paragraph{Case: $n$ is odd.} We know that
  \[
    \abs{\cS_n(\floor{n/2})} = \binom{n}{(n-1)/2} \quad\text{and}\quad \abs{\B_n(\floor{n/2})} = 2^{n-1}.
  \]
  Since $\binom{n}{(n-1)/2} = \frac12\binom{n+1}{(n+1)/2} \ge 2^n/\sqrt{2(n+1)}$ and $n \ge 3$, we get 
  \begin{equation*}
    \frac{\abs{\cS_n(\floor{n/2})}}{\abs{\B_n(\floor{n/2})}} = \frac{\binom{n}{(n-1)/2}}{2^{n-1}} \ge \frac{2}{\sqrt{2(n+1)}} \ge \frac{1}{\sqrt{n}}. \qedhere
  \end{equation*}
\end{proof}

\section{Verification of \texorpdfstring{\eqref{eqn:check-later-1}}{(17)} and \texorpdfstring{\eqref{eqn:check-later-2}}{(18)}} \label{sec:inequalities}

\begin{proof}[Proof of \eqref{eqn:check-later-1}]
Eliminating the denominators, \eqref{eqn:check-later-1} is equivalent to the following inequalities
\[
  sc(rc_0 - t) < (t+r)(nc_0 - sc) \quad\Longleftrightarrow\quad rsc(c_0 + 1) < (t+r)nc_0.
\]
Recall $c < c_0$ from \eqref{eqn:aux_a}. Because $c(c_0 + 1) < c_0(c+1)$, it suffices to check
\[
  rs(c+1) \le (t+r)n \quad\Longleftrightarrow\quad c \le t\left(\frac{1}{r}+\frac{1}{s}\right) + \frac{r}{s}.
\]
Recall that
\[
  c = \sqrt{\frac{n}{rs}} = \sqrt{\frac1r + \frac1s}
  \quad\text{and}\quad
  t = \frac{s}{r+1} - \frac{r}{s+1} \ge 0.
\]
It suffices to check the following is non-negative:
\[
  \left(\left( \frac{s}{r+1} - \frac{r}{s+1} \right) \left( \frac1r + \frac1s \right) + \frac{r}{s}\right)^2 - \left( \frac1r + \frac1s \right),
\]
which after multiplying $r^2(r+1)^2s(s+1)^2$ equals
\begin{gather*}
  (s-r)^5 \\
  + (7r+2)(s-r)^4 \\
  + (r^3 + 17r^2 + 9r + 1)(s-r)^3 \\
  + r(4r^3 + 17r^2 + 10r + 1)(s-r)^2 \\
  + r(r+1)^2(r^3 + 3r^2 - 3r - 1)(s-r) \\
  + (r-2)r^2(r+1)^4,
\end{gather*}
which clearly is non-negative for $r \ge 2$ and $s \ge r$.
\end{proof}

\begin{proof}[Proof of \eqref{eqn:check-later-2}]
  After expanding both sides a bit, \eqref{eqn:check-later-2} is equivalent to
  \begin{multline*}
    \frac{s^2}{4r^2} + \left(\frac{t+r}{rc} - \frac{1}{c_0} + \frac{s}{r} \right)\left(\frac{t+r}{rc} - \frac{1}{c_0} \right) < \left( \frac{s}{r} + \frac{r(t+r)}{c} + \frac{s}{rc_0} \right) \left( \frac{t+r}{rc} - \frac{1}{c_0} \right) \\
    \Longleftrightarrow\quad
    \frac{s^2}{4r^2} < \left( \left( r+\frac1r \right)\frac{t+r}{c} + \frac{r+s}{rc_0} \right)\left( \frac{t+r}{rc} - \frac{1}{c_0} \right),
  \end{multline*}
  which after multiplying both sides by $r^2$ and expanding the right hand side is equivalent to
  \begin{multline*}
    \frac{s^2}{4} < \frac{(r^2+1)(t+r)^2}{c^2} + \frac{(s-r^3)(t+r)}{cc_0} - \frac{r(r+s)}{c_0^2} \\
    = \frac{(r^2t+t+r)(t+r)}{c^2} + \frac{s(t+r)}{cc_0} + \frac{r^3(t+r)}{c}\left( \frac{1}{c}-\frac{1}{c_0} \right) - \frac{r(r+s)}{c_0^2}.
  \end{multline*}
  Using $c < c_0$ from \eqref{eqn:aux_a}, it suffices to check the following is positive:
  \[
    \frac{(r^2t+t+r)(t+r)}{c^2} + \frac{s(t+r)}{c_0^2} - \frac{r(r+s)}{c_0^2} - \frac{s^2}{4} = \frac{(r^2t+t+r)(t+r)}{c^2} + \frac{st-r^2}{c_0^2} - \frac{s^2}{4},
  \]
  which after substituting $c^2 = (r+s)/(rs)$ and $c_0^2 = (r+s-1)/(r(s-1))$ equals
  \[
    \frac{(r^2t+t+r)(t+r)rs}{r+s} + \frac{(st-r^2)r(s-1)}{r+s-1} - \frac{s^2}{4},
  \]
  which after multiplying $4(r+1)^2(s+1)^2(r+s)(r+s-1)$ equals
  \begin{gather*}
    (4r^3 + 3r^2 + 6r - 1)(s-r)^6 \\
    + (4 r^5+32 r^4+32 r^3+51 r^2-2 r-1) (s-r)^5 \\
    + (24 r^6+100 r^5+122 r^4+160 r^3+3 r^2-14 r+1)(s-r)^4 \\
    + (52 r^7+152 r^6+208 r^5+232 r^4-12 r^3-59 r^2-2 r+1)(s-r)^3 \\
    + r(48 r^7+112 r^6+163 r^5+156 r^4-48 r^3-96 r^2-11 r+4) (s-r)^2 \\
    + r^2(r+1)^2(16 r^5+32 r^3-23 r^2-14 r+5)(s-r) \\
    + 2r^3(r+1)^4,
  \end{gather*}
  which clearly is positive for $r \ge 2$ and $s \ge r$.
\end{proof}

\end{document}